\documentclass{article}
\usepackage{amsmath, amssymb}
\usepackage{graphicx}
\usepackage{amsmath, amssymb, amsthm}
\usepackage{mathrsfs}
\usepackage{xcolor, hyperref}
\usepackage[dvipsnames]{xcolor}
\usepackage[margin=1in]{geometry}
\newtheorem{theorem}{Theorem}
\newtheorem{lemma}[theorem]{Lemma}

\newtheorem*{remark}{Remark}
\newtheorem*{conjecture}{Conjecture}

\title{Two new proofs of Chui's Conjecture in Weighted Bergman Spaces}
\author{Georgia Corbett}

\begin{document}

\maketitle

\begin{abstract}
Chui's conjecture asks whether the average electrostatic field generated by $N$ unit point charges on the unit circle is minimized when the charges are equally spaced. Abakumov, Borichev, and Fedorovskiy proved an analogue of this conjecture in weighted Bergman spaces, showing that the corresponding norm is minimized uniquely by the $N$th roots of unity. 

In this paper, we give two alternative approaches to this weighted minimization problem. The first is a direct calculus-based proof. We reduce the norm to a pairwise interaction energy, identify the equally spaced configuration as a critical point, and analyze the resulting gradient system and Hessian. The second approach connects the problem with discrete energy minimization on spheres by Cohn and Kumar. By rewriting the Bergman interaction as a completely monotonic function of squared Euclidean distance, we recover the minimizing property of the regular $N$-gon. This second argument also extends the minimization result beyond the concave-weight setting of Abakumov, Borichev, and Fedorovskiy, including the power weights $g_\alpha(t)=t^\alpha$ for every $\alpha>0$.
\end{abstract}
\section{Introduction}

In 2020, Abakumov, Borichev, and Fedorovskiy studied a weighted Bergman space analogue of Chui's conjecture. For a fixed positive integer $N$, consider a simplest fraction of the form
\[
\sum_{0\leq k<N}\frac{1}{z-a_k},
\]
where the poles $a_k\in\mathbb{C}$. Let $g$ be a positive, integrable function on $[0,1]$. The corresponding generalized weighted Bergman space is
\[
A^2_{(g)}(\mathbb{D})=\left\{f\in\operatorname{Hol}(\mathbb{D}):\|f\|_{(g)}^2=K_g\iint_{\mathbb{D}}|f(z)|^2 g(1-|z|^2)\,dA(z)<\infty\right\},
\]
where $K_g=\left(\int_0^1 g(t)\,dt\right)^{-1}.$ Restricting the poles $a_k$ to the unit circle $\mathbb{T}$, Abakumov, Borichev, and Fedorovskiy asked which configuration of the $N$ poles minimizes the corresponding weighted Bergman norm. Abakumov, Borichev, and Fedorovskiy main result is as follows

\begin{theorem}[Abakumov, Borichev, Fedorovskiy]
\label{ABF Theorem 1}
Let $g\not\equiv 0$ be a concave, non-decreasing function on $[0,1]$
satisfying $g(0)=0$ and
\begin{equation}\label{g condition}
\int_0^1 \frac{g(s)}{s}\,ds<\infty.
\end{equation}

Then, for every integer $N\geq 1$ and every family of points
$\{a_k\}_{0\leq k<N}$ on the unit circle,
\[
\left\|
\sum_{0\leq k<N}
\frac{1}{z-a_k}
\right\|_{(g)}
\geq
\left\|
\sum_{0\leq k<N}
\frac{1}{z-e^{2\pi i k/N}}
\right\|_{(g)}.
\]
Furthermore, if $\{a_k\}_{0\leq k<N}$ minimizes the weighted Bergman norm,
then the points are equispaced on $\mathbb{T}$.
\end{theorem}

Their proof makes essential use of the Hilbert space structure of the weighted Bergman space together with convexity arguments \cite{ABF2021}.

The purpose of this paper is to give two substantially different approaches to this minimization problem. The first is a direct, calculus-based proof of Theorem \ref{ABF Theorem 1} when $g(t)=t^\alpha$ for $\alpha\in (0,1]$. We formulate the problem as a multivariable optimization problem, analyze the resulting coupled gradient system, and study the Hessian at the equispaced configuration. 
Our second approach connects the weighted Bergman problem with the theory of optimal point configurations on spheres. By rewriting the interaction between the poles as a pairwise energy depending on squared Euclidean distance, we apply the universal optimality theorem of Cohn and Kumar \cite{cohnKumar2007}. In addition to giving a second proof of the minimizing property of the equally spaced configuration, this argument extends the range of weights for which the minimization conclusion holds. In particular, for the power weights $g_\alpha(t)=t^\alpha$, the Cohn, Kumar approach yields the minimizing property for every $\alpha>0$, rather than only the range $0<\alpha\leq1$ covered by the convexity argument of Abakumov, Borichev, and Fedorovskiy.

It is important to note that the weighted Bergman problem studied by Abakumov, Borichev, and Fedorovskiy originates from a natural physical question posed by Chui: how should one distribute $N$ point charges along the unit circle $\mathbb{T}$ to minimize the average strength of the resulting electrostatic field within the unit disk $\mathbb{D}$? In 1971, Chui conjectured that this average field strength reaches its absolute minimum when the point charges are equally spaced along the boundary \cite{chui1971}.
\begin{conjecture}[Chui's Conjecture]
For any positive integer $N$, and for any family of points $\{a_k\}_{0\leq k<N}$ on the unit circle,
\[\left\|\sum_{0\leq k<N}\frac{1}{z-a_k}\right\|_{L^1(\mathbb{D})}\geq\left\|\sum_{0\leq k<N}\frac{1}{z-e^{2\pi i k/N}}\right\|_{L^1(\mathbb{D})}.\]
\end{conjecture}

Chui's conjecture also has a deep history with the study of approximation properties of simplest fractions. Approximation by simplest fractions is also closely connected to polynomial approximation and the distribution of zeros. Chui originally framed his conjecture in the context of the density of simplest fractions with poles on the unit circle in the Bergman space $A^1(\mathbb{D})$. Shortly after the conjecture was posed, Newman proved in 1972 that this family is not dense in $A^1(\mathbb{D})$ by establishing a uniform lower bound \cite{Newman1972}. The connection between point distributions, constrained polynomial approximation, and simplest fractions predates Chui's conjecture; see, for example, Korevaar \cite{Korevaar1964}.

Despite this elementary formulation, Chui's original conjecture remained open for over fifty years. Recently, Bezrukov gave a proof of a generalized $L^1$ result in a September 2026 preprint, together with a generalization to a broader class of functionals, including weighted $L^p$ norms with radially symmetric weights \cite{Bezrukov2026}. This development places the weighted Bergman minimization problem in a broader family of extremal problems for simplest fractions and provides a fourth proof of the minimizing result of Abakumov, Borichev, and Fedorovskiy. Further, related work by Doubtsov, Tselishchev, and Vasilyev studies a weighted version of Chui's conjecture for arbitrary positive charges on the boundary of the unit ball, proving an analogue of Newman's lower bound in higher dimensions and showing that the estimate is sharp in the two-dimensional case \cite{DoubtsovTselishchevVasilyev2026}.

\subsection{The Radial Weight Setting and Convexity}

For our first proof, we specialize Theorem~\ref{ABF Theorem 1} to the standard weighted Bergman spaces
$A^2_\alpha(\mathbb{D})$, obtained by taking
\[g_\alpha(t)=t^\alpha.\]
In this case,  $K_{g_\alpha}=\left(\int_0^1 t^\alpha\,dt\right)^{-1}=\alpha+1,$ and hence
\[
\|f\|_\alpha^2
=
(\alpha+1)
\iint_{\mathbb{D}}
|f(z)|^2
(1-|z|^2)^\alpha\,dA(z).
\]

The condition \eqref{g condition} becomes
\[
\int_0^1 s^{\alpha-1}\,ds<\infty,
\]
which holds precisely when $\alpha>0$. Thus simplest fractions with poles on $\mathbb{T}$ belong to $A^2_\alpha(\mathbb{D})$ for every $\alpha>0$ \cite{ABF2021}.

For the radial weight $g_\alpha(t)=t^\alpha$, define the function
\[
\phi_\alpha(\theta)
:=
\int_0^1
\frac{\cos\theta-s}
{1+s^2-2s\cos\theta}
(1-s)^\alpha\,ds.
\]
A key ingredient in the proof of Abakumov, Borichev, and Fedorovskiy is the strict convexity of this interaction function. We will use the following specialization of Lemma 12 in \cite{ABF2021} throughout our calculus-based proof.

\begin{lemma}[Abakumov, Borichev, Fedorovskiy]
\label{ABF convexity lemma}
For the weight $g_\alpha(t)=t^\alpha$, $\alpha>0$, the
    corresponding function $\phi_\alpha$ is strictly convex on $(0,2\pi)$ if and only if $\alpha\in(0,1].$

   
\end{lemma}

\section{Proof of Theorem \ref{ABF Theorem 1} when $\alpha\in (0,1]$}


We will prove Theorem \ref{ABF Theorem 1} when $\alpha\in (0,1]$. The proof proceeds in three main steps. We first use the Hilbert space structure of the weighted Bergman space to reduce the norm minimization problem to the study of a sum of pairwise interactions between the points on the unit circle. We then show that the equally spaced configuration, corresponding to the $N^{\text{th}}$ roots of unity, is a critical point of this interaction sum. Finally, we prove that it is the only critical point. For this uniqueness argument, we consider the cases of odd and even $N$ separately. In each case, the critical point equations are encoded as a matrix system, and an analysis of the kernel of the resulting matrix allows us to conclude that the points must be equally spaced.

Let
$$f(z) = \sum_{k=0}^{N-1} \frac{1}{z - e^{i\gamma_k}},$$
After a rotation, we may assume that $0=\gamma_0\leq\gamma_1\leq\dots\leq\gamma_{N-1}$.

Note that simplest fractions such as $f(z)$ above are not in $A^2_0$. See \cite{ABF2021}. So assume $\alpha > 0$. We will show that $\|f\|_\alpha^2$ is minimized when the points are equally spaced, equivalently, when they are the $N$th roots of unity.

First recall that in a Hilbert space, we can expand $\|f(z)\|^2_\alpha$ via the orthonormal basis of the monomials $\{z^n\}_{n=0}^\infty$. Using this expansion, we have:
\[
\|f(z)\|^2_\alpha = N \left\|\frac{1}{z-1}\right\|^2_{\alpha}+ 2\pi \;(\alpha+1)\sum_{0 \leq k < j \leq N-1} \phi_\alpha (\gamma_j - \gamma_k)
\]
where
\[
\phi_\alpha (\gamma_j - \gamma_k) = \int_0^1
\frac{\cos(\gamma_j - \gamma_k)- s}
{1 + s^2 - 2s \cos(\gamma_j - \gamma_k)}(1 - s)^\alpha \, ds.
\]

Notice for $0 < \alpha \leq 1$, $\phi_\alpha(\theta)$ is $C^\infty((0,2\pi))$. See \cite{ABF2021}. 

\begin{remark}
    It is important to note that $$\phi_\alpha=\sum_{n\geq 0} \cos((n+1)\theta)\|z^n\|_\alpha^2= \int_0^1 \frac{\cos(\theta)- s}{1+s^2-2s\cos(\theta)}(1-s)^\alpha ds.$$ 
    We use the integral form of $\phi_\alpha$ as taking derivatives of the sum forces stronger convergence conditions than needed. 
\end{remark}

So our problem reduces to minimizing
$$ \sum_{0 \leq k < j \leq N-1} \phi_\alpha(\gamma_j - \gamma_k).$$

Denote
$$G(\gamma_1, \dots, \gamma_{N-1}) = \sum_{0 \leq k < j < N} \phi_\alpha(\gamma_j - \gamma_k).$$

We claim the $N^{th}$ roots of unity are critical points:
$$\nabla G\left(\frac{2\pi}{N}, \dots, \frac{2\pi (N-1)}{N}\right) = (0, \dots, 0).$$

First, observe that $\phi_\alpha$ is even and $2\pi$-periodic, since it
depends on $\theta$ only through $\cos\theta$. Hence
\[
\phi_\alpha'(2\pi-\theta)
=
-\phi_\alpha'(\theta).
\]

For $m\in\{1,\dots,N-1\}$, differentiating $G$ with respect to
$\gamma_m$ gives
\[
\frac{\partial G}{\partial\gamma_m}
=
\sum_{\ell\neq m}
\phi_\alpha'(\gamma_m-\gamma_\ell),
\]
where the differences are modulo $2\pi$.

At the equally spaced configuration
\[
\gamma_j=\frac{2\pi j}{N},
\qquad j=0,\dots,N-1,
\]
we obtain
\[
\frac{\partial G}{\partial\gamma_m}
=
\sum_{\ell\neq m}
\phi_\alpha'\left(
\frac{2\pi(m-\ell)}{N}
\right).
\]
As $\ell$ ranges over $\{0,\dots,N-1\}\setminus\{m\}$, the quantity
$m-\ell$ modulo $N$ ranges over $1,\dots,N-1$. Therefore
\[
\frac{\partial G}{\partial\gamma_m}
=
\sum_{r=1}^{N-1}
\phi_\alpha'\left(\frac{2\pi r}{N}\right).
\]

We now pair the term indexed by $r$ with the term indexed by $N-r$.
Using
\[
\phi_\alpha'\left(
\frac{2\pi(N-r)}{N}
\right)
=
\phi_\alpha'\left(
2\pi-\frac{2\pi r}{N}
\right)
=
-\phi_\alpha'\left(
\frac{2\pi r}{N}
\right),
\]
each pair cancels. If $N$ is even, the only unpaired term is $N/2$, but
\[
\phi_\alpha'(\pi)
=
-\phi_\alpha'(\pi),
\]
so $\phi_\alpha'(\pi)=0$. Thus, in either case,
\[
\frac{\partial G}{\partial\gamma_m}=0.
\]

Since this holds for every $m=1,\dots,N-1$, we conclude that
\[\nabla G\left(
\frac{2\pi}{N},\dots,\frac{2\pi(N-1)}{N}
\right)=\mathbf 0.\]

It remains to show that this is the only critical point.

\subsection{$N$ odd}
To prove uniqueness of the critical point, we consider the cases of odd and even $N$ separately. Both cases follow the same general strategy. We first rewrite the critical-point equations as a matrix equation
\[
A\mathbf{c}=\mathbf{0}.
\]
A key observation is that we recover a block structure of our matrix that corresponds to the step distance of the interaction of points on the circle. So we analyze the kernel of $A$ and use the strict monotonicity of $\phi_\alpha'$ to show that the only possible critical configuration is the equally spaced one.

Throughout this subsection, let
$N=2M+1$ and $M=\left\lfloor \frac{N}{2}\right\rfloor=\frac{N-1}{2}.$
After a rotation, we may assume $\gamma_0=0$, and we write
\[
0=\gamma_0<\gamma_1<\cdots<\gamma_{N-1}<2\pi.
\]
Define the consecutive gaps
\[
x_k=\gamma_k-\gamma_{k-1},
\qquad\text{ for } k=1,\dots,N-1,
\]
and
\[
x_N=2\pi-\gamma_{N-1}.
\]
Thus
\[
x_1+\cdots+x_N=2\pi.
\]
All indices involving the $x_k$ will be understood modulo $N$.

\begin{lemma}
    Let $N\in \mathbb{N}$ be odd and let $\{e^{i\gamma_k}\}_{0\leq k < N} \subset \mathbb{T}.$ Then   the system from     
    \[
\nabla\left \|\sum_{0 \leq k < N} \frac{1}{z - e^{i \gamma_k}}\right\|^2_\alpha = \nabla \left( \sum_{0 \leq k < j \leq N-1} \phi_\alpha (\gamma_j - \gamma_k)\right)=0\]

 where $$\phi_\alpha (\gamma_j - \gamma_k) = \int_0^1
\frac{\cos(\gamma_j - \gamma_k)- s}
{1 + s^2 - 2s \cos(\gamma_j - \gamma_k)}(1 - s)^\alpha \, ds, $$
has the matrix representation $A\mathbf{c}=\mathbf{0}$ where

$$A=\left[ I - T_N \mid T_N^{-1} - T_N \mid T_N^{-2} - T_N \mid \dots \mid T_N^{-(M-1)} - T_N \right],$$

$$T_N= \begin{bmatrix} 0 & 1 & 0 & \dots & 0 \\ 0 & 0 & 1 & \dots & 0 \\ \vdots & & \ddots & & \vdots  \\ 0 & 0 & \dots & 0 & 1 \\ 1 & 0 & \dots & 0 & 0 \end{bmatrix} \in M_{N\times N}(\mathbb{R}) \text{ and } $$
$$\mathbf{c}= \left[c_{1,1}, c_{2,1}, \dots, c_{N,1} \mid c_{1,2}, c_{2,2}, \dots, c_{N, 2} \mid \dots \mid c_{1,M}, \dots, c_{N,M}\right]^T =[\mathbf{c}_1 \mid \mathbf{c}_2 \mid \dots \mid \mathbf{c}_M]^T.$$
For $1\leq m\leq M$, define
\[
c_{k,m}
=
\phi_\alpha'
\left(
x_k+x_{k+1}+\cdots+x_{k+m-1}
\right),
\]
and let
\[
\mathbf{c}_m
=
[c_{1,m},c_{2,m},\dots,c_{N,m}]^T.
\]
Then
\[
\mathbf{c}
=
[\mathbf{c}_1\mid\mathbf{c}_2\mid\cdots\mid\mathbf{c}_M]^T.
\]
\end{lemma}


\begin{proof}

Since $G$ depends only on differences of the angles, it is invariant under simultaneous rotation. Hence
\[
\sum_{i=0}^{N-1}\frac{\partial G}{\partial \gamma_i}=0.
\]

After fixing $\gamma_0=0$, the critical-point equations are those corresponding to
\[
i=1,\dots,N-1.
\]
We adjoin the redundant equation
\[
\frac{\partial G}{\partial \gamma_0}
=
-\sum_{i=1}^{N-1}\frac{\partial G}{\partial \gamma_i}
=
0,
\]
which allows us to write the system in the cyclic matrix form
\[
A\mathbf{c}=\mathbf{0}.
\]
 
Then
$$\frac{\partial}{\partial \gamma_i}\left(\sum_{0\leq k<j\leq N-1} \phi_\alpha (\gamma_j-\gamma_k)\right) =\sum_{0\leq k\leq N-1} \varepsilon_i(k) \phi_\alpha'(\varepsilon_i(k)(\gamma_i-\gamma_k))$$
 
where $\varepsilon_i(k)= \begin{cases} 1 & i>k \\ -1 & i < k \end{cases}$. For a fixed step distance $m\in \{1,2,\dots, M\}$ , we have two interacting nodes with $\gamma_i$: counterclockwise $m$ steps, $\gamma_{i+m}$, and clockwise $m$ steps, $\gamma_{i-m}$.

$$\text{If } i < i+m \pmod{N} \hspace{0.25in} \text{ then } \hspace{0.25in}-\phi_\alpha  '(\gamma_{i+m}-\gamma_i)= -\phi_\alpha '(x_{i+1}+\dots+ x_{i+m})= -c_{i+1,m}.$$ 

$$\text{If } i > i+m \pmod{N}, \hspace{0.25in} \text{ then } \hspace{0.25in} \phi_\alpha  '(\gamma_i-\gamma_{i+m})=-\phi_\alpha'(2\pi -(\gamma_{i+m} - \gamma_i))= -c_{i+1,m}$$
since $\phi_\alpha'(2\pi-\theta)=-\phi_\alpha'(\theta)$.

$$\text{If } i > i-m \pmod{N} \hspace{0.25in} \phi_\alpha  '(\gamma_i-\gamma_{i-m})=\phi_\alpha '(x_{i-m+1}+\dots+x_{i})=c_{i-m+1,m}$$

$$\text{If } i < i-m \pmod{N} \hspace{0.25in} -\phi_\alpha  '(\gamma_i-\gamma_{i-m})=-(-\phi_\alpha'(2\pi -(\gamma_{i-m}-\gamma_i))=c_{i-m+1,m}$$

Then to get to the $i$th equation:
$$\frac{\partial}{\partial \gamma_i}\|f\|_\alpha^2=\sum_{m=1}^M (c_{i-m+1,m} - c_{i+1,m}) = 0$$

Now notice that for $\mathbf{c}=[\mathbf{c}_1 \mid \mathbf{c}_2 \mid \dots \mid \mathbf{c}_M]^T$, 
 we have that $T_N$ shifts the $i^{th}$ spot in the $m^{th}$ block  forwards by 1:
$$\left(T_N \mathbf{c}_{m}\right)_i = c_{(i+1) \text{mod } N,m},$$

 Similarly, we have $T_N^{-(m-1)}$ shifts  the $i^{th}$ spot in the $m^{th}$ block  backwards by $m-1$.
$$\left (T_N^{-(m-1)}\mathbf{c}_m \right )_i = \mathbf{c}_{i-m+1,m}.$$

So on the $m^{th}$ block, $\mathbf{c}_m$, we have 
$$c_{i-m+1,m} - c_{i+1,m}= \left(\left( T_N^{-(m-1)}-T_N \right )\left (\mathbf{c}_m\right )\right)_i$$

$$(A\mathbf{c})_i=\sum_{m=1}^M \left((T_N^{-(m-1)}-T_N)\mathbf{c}_m \right)_i= 0$$

Therefore, $A$ is an $N\times MN$ matrix given by
$$A\mathbf{c}= \left[I - T_N \mid T_N^{-1} - T_N \mid \dots \mid T_N^{-(m-1)} - T_N \mid \dots \mid T_N^{-(M-1)} - T_N \right] \begin{bmatrix} \mathbf{c}_1\\ \mathbf{c}_2\\ \vdots\\ \mathbf{c}_M\\ \end{bmatrix}=\mathbf{0}.$$
\end{proof}


\subsubsection{Kernel Analysis}

The matrix equation $A\mathbf{c}=\mathbf{0}$ encodes the critical-point equations for the configuration. Thus, any critical configuration determines a vector $\mathbf{c}\in\ker(A)$. To characterize the possible critical points, we therefore study the structure of $\ker(A)$. Using an explicit basis of the kernel together with the strict monotonicity of $\phi_\alpha'$, we will show that the only configuration whose associated vector $\mathbf{c}$ lies in $\ker(A)$ is the equally spaced configuration. Consequently, the $N$th roots of unity form the unique critical configuration. 

We proceed in three steps. First, we determine the dimension of $\ker(A)$. We then exhibit an explicit basis for the kernel by showing that the proposed vectors lie in $\ker(A)$ and are linearly independent. Finally, we use this basis representation together with the monotonicity of $\phi_\alpha'$ to obtain restrictions on the gaps $x_1,\dots,x_N$.

\begin{lemma}
Let
\[A=\left[I_N-T_N\;\middle|\;T_N^{-1}-T_N\;\middle|\;\cdots\;\middle|\;T_N^{-(M-1)}-T_N\right]. \]
Then
\begin{equation}\label{n odd dim}
    \dim\ker(A)=(M-1)N+1.
\end{equation}

We claim that
\begin{equation}\label{N odd span vectors}
   \ker(A)
=
\operatorname{span}\{V_0,V_1,\dots,V_{(M-1)N}\}. 
\end{equation}

We define
\[
V_0=
[\mathbf{1}\mid\mathbf{0}\mid\cdots\mid\mathbf{0}],
\]
where
\[
\mathbf{1}=(1,\dots,1)^T,
\]
together with the vectors $V_{(m-2)N+k}$, where
\[
m=2,\dots,M,
\qquad
k=1,\dots,N,
\]
defined blockwise by
\[
V_{(m-2)N+k}
=
[\mathbf{v}_{(m-2)N+k}^{(1)}\mid\cdots\mid \mathbf{v}_{(m-2)N+k}^{(M)}],
\]
with
\[
\mathbf{v}_{(m-2)N+k}^{(1)}
=
-\sum_{r=0}^{m-1}\mathbf{e}_{k+r},\]
and 
\[\mathbf{v}_{(m-2)N+k}^{(m)}=\mathbf{e}_k,
\]
and
\[
\mathbf{v}_{(m-2)N+k}^{(r)}=\mathbf{0}
\qquad
\text{for }r\neq 1,m \text{ where } \mathbf{e}_k \text{ is the standard Euclidean basis.}
\]
Here all indices are interpreted cyclically modulo $N$.

   
   

\end{lemma}

\begin{proof} 

\noindent\textbf{Proof of \eqref{n odd dim}:}

We first determine the dimension that our proposed basis must have. First recall that the last row is a linear combination of the preceding rows. So
\[
\text{rank}(A)\leq N-1.
\]
Further, observe that $A$ contains $I_N-T_N$ as its first block. Since $\text{rank}(I_N-T_N)=N-1,$ it follows that $\text{rank}(A)\geq N-1.$
Therefore,
\[
N-1\leq \operatorname{rank}(A)\leq N-1,
\]
and hence
\[
\operatorname{rank}(A)=N-1.
\]


Since $A$ has $MN$ columns, the rank-nullity theorem gives
\[
\dim\ker(A)
=
MN-(N-1)
=
(M-1)N+1.
\]
Thus it is enough to exhibit $(M-1)N+1$ linearly independent vectors in $\ker(A)$.

\noindent\textbf{Proof of \eqref{N odd span vectors}:}
We now verify that the proposed vectors belong to the kernel. Clearly,
\[
AV_0=(I_N-T_N)\mathbf{1}=\mathbf{0}
\]
since $T_N$ fixes the constant vector. Now consider $V_{(m-2)N+k}$. By construction, the only nonzero blocks of this vector are the first and the $m$th blocks. Therefore,
\[
AV_{(m-2)N+k}
=
(I_N-T_N)\mathbf{v}_{(m-2)N+k}^{(1)}
+
(T_N^{-(m-1)}-T_N)\mathbf{v}_{(m-2)N+k}^{m}.
\]

Using the construction of the first block,
\[
\mathbf{v}_{(m-2)N+k}^{(1)}
=
-\mathbf{e}_k-\mathbf{e}_{k+1}-\dots-\mathbf{e}_{k+m-1},
\]
we obtain
\begin{align*}
(I_N-T_N)\mathbf{v}_{(m-2)N+k}^{(1)}
&=
(I_N-T_N)
(-\mathbf{e}_k-\mathbf{e}_{k+1}-\dots-\mathbf{e}_{k+m-1})\\
&= \mathbf{e}_{k-1}-\mathbf{e}_{k+m-1}.
\end{align*}

Also,
\[\begin{aligned}
(T_N^{-(m-1)}-T_N)v_{(m-2)N+k}^{(m)} &= (T_N^{-(m-1)}-T_N)e_k\\
&= e_{k+m-1}-e_{k-1}.
\end{aligned}\]
Therefore,
\[AV_{m,k}=\mathbf{0}\]
so every proposed vector belongs to $\ker(A)$.




It remains to show that these kernel vectors are linearly independent.

Suppose
\[
\sum_{p=0}^{(M-1)N} a_pV_p=\vec{0}.
\]

By construction, for each block $m=2,\dots,M$,
\[
\mathbf{v}_{(m-2)N+k}^{(m)}=\mathbf{e}_k,
\qquad k=1,\dots,N,
\]
and the other proposed basis vectors associated with different blocks are zero in the $m$th block.

Since $\mathbf{e}_1,\dots,\mathbf{e}_N$ are linearly independent,
\[
a_1=a_2=\dots=a_{(M-1)N}=0.
\]

This simplifies Row 1 to just $a_{0}=0$.
Thus
\[
\{V_0,V_1,\dots,V_{(M-1)N}\}
\]
is linearly independent.
\end{proof}

\subsubsection{Uniqueness of the Critical Configuration}

We now use this explicit basis to translate the condition $\mathbf{c}\in\ker(A)$ back into conditions on the gaps $x_1,\dots,x_N$.

Take any element in $\ker(A)$:
\[
a_{0}v_{0} + \sum_{m=2}^{M} \sum_{k=1}^{N} a_{(m-2)N+k}v_{(m-2)N+k}
\]
By construction, any $m \geq 2$, $1 \leq k \leq N$:
\[
a_{(m-2)N+k} = \phi_\alpha '(x_{k}+\dots+x_{k+m-1})
\]
because the $k$th coordinate of the $m$th block is contributed only by the vector $V_{(m-2)N+k}$.


Notice that the first block of $\mathbf{c}$ determines the one-step interaction terms $\phi_\alpha'(x_k)$ in terms of these coefficients.
In particular,
\begin{align*}
\phi_\alpha '(x_{1}) &= a_{0} - \left[\sum_{m=2}^{M} a_{(m-2)N+1} + a_{(m-2)N+N} + a_{(m-2)N+(N-1)} + \dots + a_{(m-2)N+N-m+2}\right] \\
\phi_\alpha '(x_{2}) &=  a_{0} - \left[\sum_{m=2}^{M} a_{(m-2)N+2} + a_{(m-2)N+1} + a_{(m-2)N+N} + \dots + a_{(m-2)N+N-m+3}\right] \\
\phi_\alpha '(x_{N}) &=  a_{0} - \left[\sum_{m=2}^{M} a_{(m-2)N+N} + a_{(m-2)N+N-1} + \dots + a_{(m-2)N+N-(m-1)}\right]
\end{align*}

\noindent We now argue by contradiction. Suppose that the gaps are not all equal, and choose $x_1$ to be a maximal gap. Note, if the gaps are not all equal, the set of maximal gaps is a proper nonempty subset, so there must be some maximal gap is cyclically adjacent to a strictly smaller one; relabel so that $x_1$ is maximal. Thus,
\[
x_1> x_i
\qquad
\text{for all }i\neq1.
\]


We compare the equation for $x_1$ first with that of $x_2$, and then with that of $x_N$. Since $\phi_\alpha'$ is strictly increasing, a strict inequality between two gaps produces the same strict inequality between their $\phi_\alpha'$-values.

Then
\begin{align*}
x_{1} > x_{2} &\implies \phi_\alpha '(x_{1}) > \phi_\alpha '(x_{2}) \\
&\implies a_{0} - \left[\sum_{m=2}^{M} a_{(m-2)N+1} + a_{(m-2)N+N} + \dots + a_{(m-2)N+N-m+2}\right] \\
&\quad > a_{0} - \left[\sum_{m=2}^{M} a_{(m-2)N+2} + a_{(m-2)N+1} + a_{(m-2)N+N} + \dots + a_{(m-2)N+N-m+3}\right] \\
\end{align*}
so
\[-\left(\sum_{m=2}^{M} a_{(m-2)N+N-m+2}\right) > -\left(\sum_{m=2}^{M} a_{(m-2)N+2}\right).\]

Then translating the condition $\mathbf{c}\in\ker(A)$, we obtain
\begin{equation}\label{odd 1}
    \sum_{m=2}^{M} \phi_\alpha '(x_{2}+\dots+x_{m+1}) > \sum_{m=2}^{M} \phi_\alpha '(x_{N-m+2}+\dots+x_{1}).
\end{equation}

Similarly, 
\begin{align*}
x_{1} > x_{N} &\implies \phi_\alpha '(x_{1}) > \phi_\alpha '(x_{N}) \\
&\implies a_{0} - \left[\sum_{m=2}^{M} a_{(m-2)N+1} + a_{(m-2)N+N} + \dots + a_{(m-2)N+N-m+2}\right] \\
&\quad > a_{0} - \left[\sum_{m=2}^{M} a_{(m-2)N+N} + a_{(m-2)N+N-1} + \dots + a_{(m-2)N+N-(m-1)}\right] 
\end{align*}

so 
\[\implies -\left(\sum_{m=2}^{M} a_{(m-2)N+1}\right) > -\left(\sum_{m=2}^{M} a_{(m-2)N+N-(m-1)}\right),\]
and thus, 
\begin{equation}\label{odd 2}
   \sum_{m=2}^{M} \phi_\alpha '(x_{N-m+1}+\dots+x_{N}) > \sum_{m=2}^{M} \phi_\alpha '(x_{1}+\dots+x_{m})
\end{equation}

\noindent Combining inequalities \ref{odd 1} and \ref{odd 2},
\[
\sum_{m=2}^{M} \phi_\alpha '(x_{2}+\dots+x_{m+1}) + \phi_\alpha '(x_{N-m+1}+\dots+x_{N}) > \sum_{m=2}^{M} \phi_\alpha '(x_{N-m+2}+\dots+x_{1}) + \phi_\alpha '(x_{1}+\dots+x_{m})
\]

On the other hand, maximality of $x_1$ directly gives, for every $m$,

$$ x_2+\cdots+x_{m+1} \leq x_1+\cdots+x_m, $$
 and

$$ x_{N-m+1}+\cdots+x_N \leq x_{N-m+2}+\cdots+x_N+x_1, $$

because \(x_{N-m+1}\leq x_1\). Strict monotonicity of \(\phi_\alpha'\) then gives the reverse inequality between the total sums, which is a contradiction. Thus the corresponding points are equally spaced on $\mathbb{T}$, so the $N$th roots of unity give the unique critical configuration.


\subsection{$N$ Even}

Throughout this subsection, let
$N=2M$. In the same way as $N$ odd case, we may assume $\gamma_0=0$ after  rotation, and we write
\[
0=\gamma_0<\gamma_1<\cdots<\gamma_{N-1}<2\pi.
\]
Define the consecutive gaps
\[
x_k=\gamma_k-\gamma_{k-1},
\qquad\text{ for } k=1,\dots,N-1,
\]
and
\[
x_N=2\pi-\gamma_{N-1}.
\]
Thus
\[
x_1+\cdots+x_N=2\pi.
\]
All indices involving the $x_k$ will be understood modulo $N$.

The argument follows the same structure as in the odd case. The only difference occurs for the $M$-step interaction. Since $N=2M$, moving $M$ steps forward and moving $M$ steps backward from a fixed point reaches the same point. Thus there is only one $M$-step interaction rather than a pair of forward and backward interactions.

\begin{lemma}
     Let $N\in \mathbb{N}$ be even and let $\{e^{i\gamma_k}\}_{0\leq k < N} \subset \mathbb{T}.$ Then   the system from     
    \[
\nabla\left \|\sum_{0 \leq k < N} \frac{1}{z - e^{i \gamma_k}}\right\|^2_\alpha = \nabla \left( \sum_{0 \leq k < j \leq N-1} \phi_\alpha (\gamma_j - \gamma_k)\right)=0\]

 where $$\phi_\alpha (\gamma_j - \gamma_k) = \int_0^1
\frac{\cos(\gamma_j - \gamma_k)- s}
{1 + s^2 - 2s \cos(\gamma_j - \gamma_k)}(1 - s)^\alpha \, ds, $$
has the matrix representation $A\mathbf{c}=\mathbf{0}$ where

\[
A =
\left[
I - T_N
\;\middle|\;
T_N^{-1} - T_N
\;\middle|\;
\cdots
\;\middle|\;
T_N^{-(M-2)} - T_N
\;\middle|\;
T_N^{-(M-1)}
\right].
\]

$$T_N= \begin{bmatrix} 0 & 1 & 0 & \dots & 0 \\ 0 & 0 & 1 & \dots & 0 \\ \vdots & & \ddots & & \vdots  \\ 0 & 0 & \dots & 0 & 1 \\ 1 & 0 & \dots & 0 & 0 \end{bmatrix} \in M_{N\times N}(\mathbb{R}), \text{ and} $$
$$\mathbf{c}= \left[c_{1,1}, c_{2,1}, \dots, c_{N,1} \mid c_{1,2}, c_{2,2}, \dots, c_{N, 2} \mid \dots \mid c_{1,M}, \dots, c_{N,M}\right]^T =[\mathbf{c}_1 \mid \mathbf{c}_2 \mid \dots \mid \mathbf{c}_M]^T.$$

For $1\leq m\leq M$, define
\[
c_{k,m}
=
\phi_\alpha'
\left(
x_k+x_{k+1}+\cdots+x_{k+m-1}
\right),
\]
and let
\[
\mathbf{c}_m
=
[c_{1,m},c_{2,m},\dots,c_{N,m}]^T.
\]
Then
\[
\mathbf{c}
=
[\mathbf{c}_1\mid\mathbf{c}_2\mid\cdots\mid\mathbf{c}_M]^T.
\]
\end{lemma}

\begin{proof}

As in the odd case, since $G$ depends only on differences of the angles, it is invariant under simultaneous rotation. Hence
\[
\sum_{i=0}^{N-1}\frac{\partial G}{\partial\gamma_i}=0.
\]
After fixing $\gamma_0=0$, the critical-point equations are those
corresponding to $i=1,\dots,N-1$. We adjoin the redundant equation
\[
\frac{\partial G}{\partial\gamma_0}
=
-\sum_{i=1}^{N-1}\frac{\partial G}{\partial\gamma_i}
=
0.
\]
Thus we may index the critical-point equations cyclically by
$i=0,\dots,N-1$. We now use the structure of these equations to obtain
the matrix representation $A\mathbf{c}=\mathbf{0}$. 

Fix $i$.

For each
\[
m=1,\dots,M-1,
\]
we proceed exactly like in the $N$ odd case.
The case $m=M$ is different. Since $N=2M$,
\[
i+M\equiv i-M\pmod N,
\]
so moving $M$ steps forward and moving $M$ steps backward reaches the
same point. Hence this interaction is counted only once.
The backward $M$-step contribution is
\[
c_{i-M+1,M}
=
\phi_\alpha'
\left(
x_{i-M+1}+\cdots+x_i
\right).
\]
Equivalently, if we write the same interaction using the forward arc,
then
\[
x_{i-M+1}+\cdots+x_i
=
2\pi-
\left(
x_{i+1}+\cdots+x_{i+M}
\right),
\]
and therefore
\[
c_{i-M+1,M}
=
-\phi_\alpha'
\left(
x_{i+1}+\cdots+x_{i+M}
\right)
=
-c_{i+1,M}.
\]
Thus either orientation gives the same single $M$-step contribution.
Consequently, the $i$th critical-point equation is
\[
\frac{\partial}{\partial\gamma_i}\|f\|_\alpha^2
=
\sum_{m=1}^{M-1}
\left(
c_{i-m+1,m}-c_{i+1,m}
\right)
+
c_{i-M+1,M}
=
0.
\]
For $m=1,\dots,M-1$, we follow $N$ odd exactly and get
\[
c_{i-m+1,m}-c_{i+1,m}
=
\left(
(T_N^{-(m-1)}-T_N)\mathbf{c}_m
\right)_i.
\]
For the final $M$-step term,
\[
c_{i-M+1,M}
=
\left(
T_N^{-(M-1)}\mathbf{c}_M
\right)_i.
\]
Hence
\[
(A\mathbf{c})_i
=
\sum_{m=1}^{M-1}
\left(
(T_N^{-(m-1)}-T_N)\mathbf{c}_m
\right)_i
+
\left(
T_N^{-(M-1)}\mathbf{c}_M
\right)_i
=
0.
\]
Therefore,
\[
A\mathbf{c}
=
\left[
I_N-T_N
\;\middle|\;
T_N^{-1}-T_N
\;\middle|\;
\cdots
\;\middle|\;
T_N^{-(M-2)}-T_N
\;\middle|\;
T_N^{-(M-1)}
\right]
\begin{bmatrix}
\mathbf{c}_1\\
\mathbf{c}_2\\
\vdots\\
\mathbf{c}_{M-1}\\
\mathbf{c}_M
\end{bmatrix}
=
\mathbf{0}.
\]
\end{proof}

\subsubsection{Kernel Analysis}

As in the odd case, we use an explicit description of the kernel together with the strict monotonicity of $\phi_\alpha'$ to study the possible critical configurations.

We proceed in three steps. First, we determine the dimension of $\ker(A)$. We then exhibit an explicit basis for the kernel by showing that the proposed vectors lie in $\ker(A)$ and are linearly independent. Finally, we use this basis representation together with the monotonicity of $\phi_\alpha'$ to obtain restrictions on the gaps $x_1,\dots,x_N$.

\begin{lemma}
Let \[A =\left[I - T_N\;\middle|\;T_N^{-1} - T_N\middle|\;\cdots \;\middle|\; T_N^{-(M-2)} - T_N \;\middle|\; T_N^{-(M-1)} \right]. \]
Then we have
\begin{equation}\label{n even dim}
    \dim\ker(A)=MN-N=(M-1)N.
\end{equation}

We claim that
\begin{equation}\label{n even span}
    \ker(A)= \operatorname{span} \{V_0,V_1,\dots,V_{(M-1)N-1}\}.
\end{equation}

The vector $V_0$ is defined as in the odd case:
\[
V_0=
[1\;1\;\dots\;1
\mid
\mathbf{0}
\mid
\dots
\mid
\mathbf{0}]^T.
\]

For
\[
2\leq m\leq M-1,
\qquad
1\leq k\leq N,
\]
define
\[
V_{(m-2)N+k}
=
[
\mathbf{v}_{(m-2)N+k}^{(1)}
\mid
\mathbf{v}_{(m-2)N+k}^{(2)}
\mid
\dots
\mid
\mathbf{v}_{(m-2)N+k}^{(M)}
]^T,
\]
where
\[
\mathbf{v}_{(m-2)N+k}^{(1)}
=-\sum_{r=0}^{m-1} \mathbf{e}_{k+r}
\]
\[
\mathbf{v}_{(m-2)N+k}^{(m)}
=
\mathbf{e}_k,
\]
and
\[
\mathbf{v}_{(m-2)N+k}^{(p)}
=
\mathbf{0}
\qquad
\text{for all }p\neq 1,m.
\]

Lastly, define  for  $1\leq k\leq N-1$
\[
V_{(M-2)N+k}
=
[
\mathbf{v}_{(M-2)N+k}^{(1)}
\mid
\mathbf{v}_{(M-2)N+k}^{(2)}
\mid
\dots
\mid
\mathbf{v}_{(M-2)N+k}^{(M)}
]^T,
\]
where
\[
\mathbf{v}_{(M-2)N+k}^{(1)}
=
\mathbf{e}_{k+M},
\]
\[
\mathbf{v}_{(M-2)N+k}^{(M)}
=
\mathbf{e}_k-\mathbf{e}_{k+1},
\]
and
\[
\mathbf{v}_{(M-2)N+k}^{(p)}
=
\mathbf{0}
\qquad
\text{for all }p\neq 1,m .
\]
\end{lemma}

\begin{proof}
\noindent\textbf{Proof of \eqref{n even dim}:}

Following the structure of the proof of the odd case, we see that
\[
\dim\ker(A)
=
MN-N
=
(M-1)N.
\]

Thus, to characterize $\ker(A)$, we consider the proposed vectors
\[
V_0,V_1,\dots,V_{(M-1)N-1}
\]
and show that they are linearly independent and belong to $\ker(A)$.

\noindent\textbf{Proof of \eqref{n even span}:}

Our goal is to show that
\[
V_i\in\ker(A)
\qquad
\text{for all }
i\in\{0,\dots,(M-1)N-1\}.
\]

From the odd case, we already have
\[
V_0\in\ker(A).
\]

\noindent Now consider $V_{(m-2)N+k}$ for $2\leq m\leq M-1, \quad 1\leq k\leq N.$
From the block structure,
\[
A\mathbf{V}_{(m-2)N+k}
=
(I-T)\mathbf{v}_{(m-2)N+k}^{(1)}
+
(T^{-(m-1)}-T)\mathbf{v}_{(m-2)N+k}^{(m)}.
\]

From the definitions of the vectors
we obtain
\begin{align*}
A\mathbf{V}_{(m-2)N+k}
&=
(I-T)\mathbf{v}_{(m-2)N+k}^{(1)}
+
(T^{-(m-1)}-T)\mathbf{v}_{(m-2)N+k}^{(m)}\\
&=
(I-T)
\left(-\sum_{r=0}^{m-1} \mathbf{e}_{k+r}
\right)
 +
(T^{-(m-1)}-T)\mathbf{e}_k.
\end{align*}

Now notice that
\[
T\mathbf{e}_k=\mathbf{e}_{k-1}
\qquad\text{and}\qquad
T^{-p}\mathbf{e}_k=\mathbf{e}_{k+p}.
\]

Thus
\begin{align*}
A\mathbf{V}_{(m-2)N+k}
&=
(I-T)
\left(
-\mathbf{e}_k
-\mathbf{e}_{k+1}
-\dots
-\mathbf{e}_{k+(m-1)}
\right)
+
(T^{-(m-1)}-T)\mathbf{e}_k\\
&=
\left[-\mathbf{e}_k
-\mathbf{e}_{k+1}
-\dots
-\mathbf{e}_{k+(m-2)}
-\mathbf{e}_{k+(m-1)}\right]\\
&\quad
+\left[\mathbf{e}_{k-1}
+\mathbf{e}_k
+\mathbf{e}_{k+1}
+\dots
+\mathbf{e}_{k+(m-2)}\right]
+\mathbf{e}_{k+(m-1)}
-\mathbf{e}_{k-1}\\
&=
0.
\end{align*}

Lastly, consider $V_{(M-2)N+k}$ for
\[
1\leq k\leq N.
\]
Then
\begin{align*}
A\mathbf{V}_{(M-2)N+k}
&=
(I-T)\mathbf{v}_{(M-2)N+k}^{(1)}
+
T^{-(M-1)}
\mathbf{v}_{(M-2)N+k}^{(M)}\\
&=
(I-T)\mathbf{e}_{k+M}
+
T^{-(M-1)}
(\mathbf{e}_k-\mathbf{e}_{k+1})\\
&=
\mathbf{e}_{k+M}
-
\mathbf{e}_{k+M-1}
+
\mathbf{e}_{k+M-1}
-
\mathbf{e}_{k+M}\\
&=
0.
\end{align*}

Therefore, all of the proposed vectors belong to $\ker(A)$.

It remains to show that the proposed kernel vectors are linearly independent.

Consider the linear combination
\[\sum_{i=0}^{(M-1)N-1}a_iV_i\]
with coefficients $a_i\in\mathbb{R}.$

\noindent By construction, in the blocks
\[2\leq m\leq M-1,\]
the only nonzero vectors are
\[
\mathbf{v}_{(m-2)N+k}^{(m)}
=
\mathbf{e}_k.
\]
Therefore,
\[
a_1=a_2=\dots=a_{(M-2)N}=0.
\]

Notice that the remaining coupled entries of the linear combination are in the first entries and the $(M-1)N+1^{th}$ entry through $MN^{th}$ entry.

Since $a_1=a_2=\dots=a_{(M-2)N}=0,$ the first entry then reduces to
\[
a_0-a_{(M-2)N+(M+1)}=0,
\]

\[
\mathbf{v}_{(M-2)N+(M+1)}
=
\mathbf{e}_{(M+1)+M}
=
\mathbf{e}_{N+1}.
\]

Next, we show that the coefficients $a_{(M-1)N+1}$ through $a_{MN}$ have a recursive relationship coming from
\[
\mathbf{v}_{(M-2)N+k}^{(M)}
=
\mathbf{e}_k-\mathbf{e}_{k+1}
\]
forcing the coefficients to be 0.

From the definition of the last block of vectors we first have
\[ a_{(M-2)N+1}=0.\]
Then utilizing the recursive relationship we obtain
\begin{align*}
    a_{(M-2)N+2}-a_{(M-2)N+1}=0 &\quad\Rightarrow \quad a_{(M-2)N+2}=0\\
    a_{(M-2)N+3}-a_{(M-2)N+2}=0\ &\quad\Rightarrow \quad a_{(M-2)N+3}=0\\
    \vdots\\
    a_{(M-2)N+(N-1)}-a_{(M-2)N+(N-2)}=0\ &\quad\Rightarrow \quad a_{(M-2)N+(N-1)}=0\\
\end{align*}

Lastly, we see from the definition of the vectors \[a_{(M-2)N+(N-1)}=0.\]


Thus,
\[
a_0
=
a_1
=
\dots
=
a_{(M-2)N+(N-1)}
=
0.
\]

Therefore, the proposed vectors are linearly independent.
\end{proof}

\subsubsection{Uniqueness of the Critical Configuration}

We now use the kernel representation to study the possible critical point configurations.

Take
\[
\sum_{i=0}^{(M-1)N}a_iV_i
\in\ker(A).
\]

By construction,

\[\phi_\alpha'(x_k+\dots+x_{k+m-1})=a_{(m-2)N+k}
\qquad 2\leq m\leq M-1,
\qquad
1\leq k\leq N,\]
and for $1\leq k\leq N$ we have 
\[\phi_\alpha'(x_k+\dots+x_{k+M-1})
=a_{(M-2)N+k}-a_{(M-2)N+k-1}\]

We care about the one step interactions, in particular, 
\begin{align*}
\phi_\alpha'(x_1) &=a_0-\left(\sum_{m=2}^{M-1}a_{(m-2)N+1}+a_{(m-2)N+N}\right)+a_{(M-2)N+M+1}, \\
\phi_\alpha'(x_2)&=a_0-\left(\sum_{m=2}^{M-1}a_{(m-2)N+1}+a_{(m-2)N+2}\right)+a_{(M-2)N+M+2}, \\
\phi_\alpha'(x_N)&=a_0-\left(\sum_{m=2}^{M-1}a_{(m-2)N+N-1}+a_{(m-2)N+N}\right)+a_{(M-2)N+M},\\
\end{align*}

As in $N$ odd case, we proceed by contradiction. Suppose
\[
x_1>x_i
\qquad
\text{for all }i\neq1.
\]

Then,
\[
\phi_\alpha'(x_1)
>
\phi_\alpha'(x_2)
\]
and so,
\begin{align*}
&
a_0
+
a_{(M-2)N+M+1}
-
\left[
\sum_{m=2}^{M-1}
a_{(m-2)N+1}
+
a_{(m-2)N+N}
+\dots+
a_{(m-2)N+N-m+2}
\right]
\\
&>
a_0
+
a_{(M-2)N+M+2}
-
\left[
\sum_{m=2}^{M-1}
a_{(m-2)N+2}
+
a_{(m-2)N+1}
+\dots+
a_{(m-2)N+N}
\right].
\end{align*}

Canceling like terms gives
\[
a_{(M-2)N+M+1}
-
\sum_{m=2}^{M-1}
a_{(m-2)N+N-m+2}
>
a_{(M-2)N+M+2}
-
\sum_{m=2}^{M-1}
a_{(m-2)N+2}.
\]

Similarly, we have \[x_1>x_N,\] 
so then
\[
\phi_\alpha'(x_1)
>
\phi_\alpha'(x_N).
\]
Thus
\begin{align*}
&
a_0
+
a_{(M-2)N+M+1}
-
\left[
\sum_{m=2}^{M-1}
a_{(m-2)N+1}
+
a_{(m-2)N+N}
+\dots+
a_{(m-2)N+N-m+2}
\right]
\\
&>
a_0
+
a_{(M-2)N+N}
-
\left[
\sum_{m=2}^{M-1}
a_{(m-2)N+N}
+
a_{(m-2)N+N-1}
+\dots+
a_{(m-2)N+N-m+1}
\right].
\end{align*}

Again, canceling like terms gives
\[
a_{(M-2)N+M+1}
-
\sum_{m=2}^{M-1}
a_{(m-2)N+1}
>
a_{(M-2)N+M}
-
\sum_{m=2}^{M-1}
a_{(m-2)N+N-m+1}.
\]

Adding these two inequalities gives
\begin{align*}
&
a_{(M-2)N+M+1}
+
a_{(M-2)N+M+1}
-
\sum_{m=2}^{M-1}a_{(m-2)N+1}
-
\sum_{m=2}^{M-1}a_{(m-2)N+N-m+2}
\\
&>
a_{(M-2)N+M+2}
+
a_{(M-2)N+M}
-
\sum_{m=2}^{M-1}a_{(m-2)N+N-m+1}
-
\sum_{m=2}^{M-1}a_{(m-2)N+2}.
\end{align*}

Rearranging,
\begin{align*}
&
a_{(M-2)N+M+1}
-
a_{(M-2)N+M}
-
\sum_{m=2}^{M-1}
a_{(m-2)N+N-m+2}
+
a_{(m-2)N+1}
\\
&>
a_{(M-2)N+M+2}
-
a_{(M-2)N+M+1}
-
\sum_{m=2}^{M-1}
a_{(m-2)N-m+1}
+
a_{(m-2)N+2}.
\end{align*}

Rewriting in terms of $\phi_\alpha'(\cdot)$,
\begin{align*}
&
\phi_\alpha'(x_{M+1}+\dots+x_N)
-
\sum_{m=2}^{M-1}
\phi_\alpha'(x_{N-m+2}+\dots+x_{N+1})
+
\phi_\alpha'(x_1+\dots+x_m)
\\
&>
\phi_\alpha'(x_{M+2}+\dots+x_1)
-
\sum_{m=2}^{M-1}
\phi_\alpha'(x_{N-m+1}+\dots+x_N)
+
\phi_\alpha'(x_2+\dots+x_{m+1}).
\end{align*}

Rearranging once more gives
\begin{align*}
&
\phi_\alpha'(x_{N+1}+\dots+x_N)
+
\sum_{m=2}^{M-1}
\phi_\alpha'(x_{N-m+1}+\dots+x_N)
+
\phi_\alpha'(x_2+\dots+x_{m+1})
\\
&>
\phi_\alpha'(x_{N+2}+\dots+x_N+x_1)
+
\sum_{m=2}^{M-1}
\phi_\alpha'(x_{N-m+2}+\dots+x_N+x_1)
+
\phi_\alpha'(x_1+\dots+x_m).
\end{align*}

On the other hand, maximality of $x_1\geq x_i$ for all $i\neq 1,$  shows for each $2\leq m\leq M-1,$

the inequality above gives
\[
\phi_\alpha'(x_{N-m+1}+\dots+x_N)
<
\phi_\alpha'(x_{N-m+2}+\dots+x_N+x_1)
\]
and
\[
\phi_\alpha'(x_2+\dots+x_{m+1})
<
\phi_\alpha'(x_1+\dots+x_m).
\]
 Strict monotonicity of $\phi_\alpha'$ then gives the reverse inequality between the total sums, which is a contradiction. 

\subsection{Positive Definiteness of the Hessian}
We have shown that the equally spaced configuration is a unique critical point.
We now show that the Hessian of the interaction function is positive
definite throughout the configuration domain
\[\Omega= \{(\gamma_1,\dots,\gamma_{N-1}): 0<\gamma_1<\cdots<\gamma_{N-1}<2\pi\}.
\]
Since the equally spaced configuration is a critical point, strict convexity will imply that it is the unique global minimizer. We then use continuity of $G$ on $\overline{\Omega}$ to include configurations in which two or more poles coincide.
See that 

$$
H(\gamma_1, \dots, \gamma_{N-1})
=
\begin{pmatrix}
\dfrac{\partial^2 \|f\|_\alpha^2}{\partial \gamma_1^2}
&
\dfrac{\partial^2 \|f\|_\alpha^2}{\partial \gamma_1 \partial \gamma_2}
&
\dots
&
\dfrac{\partial^2 \|f\|_\alpha^2}{\partial \gamma_1 \partial \gamma_{N-1}}
\\[1em]
\dfrac{\partial^2 \|f\|_\alpha^2}{\partial \gamma_2 \partial \gamma_1}
&
\dfrac{\partial^2 \|f\|_\alpha^2}{\partial \gamma_2^2}
&
\dots
&
\dfrac{\partial^2 \|f\|_\alpha^2}{\partial \gamma_2 \partial \gamma_{N-1}}
\\
\vdots & \vdots & \ddots & \vdots
\\
\dfrac{\partial^2 \|f\|_\alpha^2}{\partial \gamma_{N-1} \partial \gamma_1}
&
\dfrac{\partial^2 \|f\|_\alpha^2}{\partial \gamma_{N-1} \partial \gamma_2}
&
\dots
&
\dfrac{\partial^2 \|f\|_\alpha^2}{\partial \gamma_{N-1}^2}
\end{pmatrix}
$$

Then
$$
\begin{aligned}
H(\gamma_1, \dots, \gamma_{N-1})
&=
\begin{pmatrix}
\displaystyle \sum_{\substack{0 \le l \le N-1 \\ l \ne 1}} \phi_\alpha''(\gamma_l-\gamma_1) & - \phi_\alpha''(\gamma_2-\gamma_1) & \dots & - \phi_\alpha''(\gamma_{N-1}-\gamma_1) \\[1em]
- \phi_\alpha''(\gamma_2-\gamma_1) & \displaystyle \sum_{\substack{0 \le l \le N-1 \\ l \ne 2}} \phi_\alpha''\bigl(\varepsilon_2(l)(\gamma_l-\gamma_2)\bigr) & \dots & - \phi_\alpha''(\gamma_{N-1}-\gamma_2) \\
\vdots & \vdots & \ddots & \vdots \\[0.5em]
- \phi_\alpha''(\gamma_{N-1}-\gamma_1) & - \phi_\alpha''(\gamma_{N-1}-\gamma_2) & \dots & \displaystyle \sum_{\substack{0 \le l \le N-1 \\ l \ne N-1}} \phi_\alpha''(\gamma_{N-1}-\gamma_l)
\end{pmatrix} \\[1.5em]
&= \begin{pmatrix}
    \phi_\alpha''(\gamma_1) & & 0 \\
    & \ddots & \\
    0 & & \phi_\alpha''(\gamma_{N-1})
\end{pmatrix} + \sum_{0<n<m<N} \phi_\alpha''(\gamma_m-\gamma_n) \mathbf{B}_{n,m}
\end{aligned}
$$

where 
$$
(\mathbf{B}_{n,m})_{i,j} = \begin{cases} 
    1 & \text{if } i=j=n \text{ or } i=j=m \\
    -1 & \text{if } (i,j)=(n,m) \text{ or } (i,j)=(m,n) \\
    0 & \text{otherwise}
\end{cases}
$$

The Hessian is positive definite at $(\gamma_1,\dots,\gamma_{N-1})$ provided that the diagonal matrix \[D(\gamma_1,\dots,\gamma_{N-1})= \begin{pmatrix} \phi_\alpha''(\gamma_1) & & 0 \\ & \ddots & \\ 0 & & \phi_\alpha''(\gamma_{N-1}) \end{pmatrix}\] is positive definite and \[\phi_\alpha''(\gamma_m-\gamma_n) \mathbf{B}_{n,m}(\gamma_1,\dots,\gamma_{N-1})\] is positive semi-definite for all $n,m$. Indeed, for $\mathbf{x}\in\mathbb{R}^{N-1}$, 

$$
\begin{aligned}
  \left \langle  \phi_\alpha''(\gamma_m-\gamma_n) \mathbf{B}_{n,m} \mathbf{x}, \mathbf{x}  \right\rangle &=  \phi_\alpha''(\gamma_m-\gamma_n)   \left \langle \mathbf{B}_{n,m} \mathbf{x}, \mathbf{x}  \right\rangle\\
  &= \phi_\alpha''(\gamma_m-\gamma_n)  \left \langle (0, \dots, x_n-x_m \text{ (nth spot)}, 0,\dots, x_m-x_n \text{ (mth spot)}, 0,\dots,0), (x_1,\dots,x_{N-1})\right\rangle\\
  &= \phi_\alpha''(\gamma_m-\gamma_n) (x_n^2-2x_nx_m +x_m^2)\\
   &= \phi_\alpha''(\gamma_m-\gamma_n) (x_n-x_m)^2\\
   &\geq 0 \\
   &\iff \phi_\alpha''(\gamma_m-\gamma_n)\geq 0.
\end{aligned}
$$

For every
\[
(\gamma_1,\dots,\gamma_{N-1})\in\Omega,
\]
we have
\[
0<\gamma_k<2\pi
\]
and, whenever $0<k<j<N$,
\[
0<\gamma_j-\gamma_k<2\pi.
\]
By Lemma~\ref{ABF convexity lemma},
\[
\phi_\alpha''(\theta)>0
\qquad
\text{for every }\theta\in(0,2\pi).
\]
Hence $D(\gamma_1,\dots,\gamma_{N-1})$ is positive definite and each
matrix
\[
\phi_\alpha''(\gamma_m-\gamma_n)B_{n,m}
\]
is positive semi-definite. Therefore
\[
H(\gamma_1,\dots,\gamma_{N-1})
\]
is positive definite for every point of $\Omega$. It follows that $G$ is strictly convex on $\Omega$. Since
\[\nabla G\left(\frac{2\pi}{N},\dots,\frac{2\pi(N-1)}{N}\right)=0,\]
the equally spaced configuration is the unique minimizer of $G$ in $\Omega$.

It remains only to consider the boundary of the configuration space. For $\alpha>0$, the interaction function extends continuously to $\theta=0$, since
\[ \phi_\alpha(0) = \int_0^1(1-s)^{\alpha-1}\,ds = \frac{1}{\alpha}. \]
Thus $G$ extends continuously to the compact simplex
\[
\overline{\Omega} =\{(\gamma_1,\dots,\gamma_{N-1}):0\leq\gamma_1\leq\cdots\leq\gamma_{N-1}\leq2\pi\}.\]
Let
\[\mathcal{C}=\left\{e^{2\pi i k/N}:k=0,\dots,N-1\right\}\]
denote the equally spaced configuration, and let
\[
\Gamma_{\mathcal C}=\left(\frac{2\pi}{N},\dots,\frac{2\pi(N-1)}{N}\right)\]
be its angle vector.

Now let $\mathcal{C}'$ be any other configuration of $N$ points on $\mathbb{T}$, with angle vector
\[
\Gamma_{\mathcal C'}\in\overline{\Omega}.
\]
Consider the line segment joining the two configurations,
\[
\Gamma(t)=\Gamma_{\mathcal C}+t(\Gamma_{\mathcal C'}-\Gamma_{\mathcal C}),\qquad0\leq t\leq1.
\]
It is easy to check that $\Omega$ is convex, so then
\[
\Gamma(t)\in\Omega
\qquad
\text{for every }0\leq t<1,
\]
even if $\mathcal{C}'$ lies on the boundary of $\Omega$. Now define
\[
h(t)=G(\Gamma(t)).
\]

Since $\mathcal{C}$ is a critical configuration, the derivative of $h$ at $0$ is
\[
h'(0)=\nabla G(\Gamma_{\mathcal C})\cdot\left(\Gamma_{\mathcal C'}-\Gamma_{\mathcal C}
\right)= 0.
\]

For $0\leq t<1$,
\[
h''(t)
=
\left(
\Gamma_{\mathcal C'}-\Gamma_{\mathcal C}
\right)^T
H(\Gamma(t))
\left(
\Gamma_{\mathcal C'}-\Gamma_{\mathcal C}
\right).
\]
Since the Hessian is positive definite throughout $\Omega$ and $\mathcal{C}'\neq\mathcal{C}$,
\[h''(t)>0.\]
Thus $h'$ is strictly increasing. Since $h'(0)=0$, it follows that
\[
h'(t) >0 \qquad \text{ for all } t\in(0,1),
\]
and so
\[
h(t)>h(0) \qquad \text{ for all } t\in(0,1).
\]
If $\mathcal{C}'$ is an interior configuration, we take $t=1$ and obtain
\[
G(\mathcal{C}')>G(\mathcal{C}).
\]
If $\mathcal{C}'$ lies on the boundary of $\overline{\Omega}$, then continuity of $G$ gives
\[
G(\mathcal{C}')=\lim_{t\to1^-}h(t).
\]
Since $h$ is strictly increasing on $(0,1)$, for any fixed
$t_0\in(0,1)$,
\[
G(\mathcal{C}') \geq h(t_0)>h(0)=G(\mathcal{C}).
\]
Therefore,
\[
G(\mathcal{C}')>G(\mathcal{C})
\]
for every configuration $\mathcal{C}'\neq\mathcal{C}$. Hence the equally spaced configuration $\mathcal{C}$ is the unique global minimizer on $\overline{\Omega}$.
\section{Alternate proof of Theorem \ref{ABF Theorem 1}}

Recall that we wish to minimize
\[
\left\|
\sum_{k=0}^{N-1}\frac{1}{z-e^{i\gamma_k}}
\right\|_{(g)}^2
=
N\left\|\frac{1}{z-1}\right\|_{(g)}^2
+
K_g
\sum_{0\leq k<j<N}
\phi_\alpha(\gamma_j-\gamma_k).
\]
Since the first term is independent of the configuration of the points, the problem reduces to minimizing the pairwise interaction energy
\[
\sum_{0\leq k<j<N}
\phi_\alpha(\gamma_j-\gamma_k).
\]
Thus, the weighted Bergman norm minimization problem may be viewed as a discrete energy minimization problem determined by the angular differences between points on the unit circle.

This second approach is motivated by the linear programming method for minimizing discrete energies on spheres. In \cite{cohnKumar2007}, Cohn and Kumar show that sharp configurations minimize the energy associated with every completely monotonic function of squared Euclidean distance. We recall a version of their Theorem 1.2 below.


\begin{theorem}[Cohn, Kumar]\label{cohnkumar theorem}
    Let $f:(0,4]\to\mathbb{R}$ be completely monotonic, and let
$C\subset S^{n-1}$ be a sharp arrangement. If $C'\subset S^{n-1}$ is any subset satisfying $|C'|=|C|$, then
\[
\sum_{\substack{x,y\in C'\\ x\neq y}}
f\bigl(|x-y|^2\bigr)
\geq
\sum_{\substack{x,y\in C\\ x\neq y}}
f\bigl(|x-y|^2\bigr).
\]

If $f$ is strictly completely monotonic, then equality implies that $C'$ is also a sharp configuration and the same distances occur in $C$ and $C'$.  In that case, if $C$ is listed in Table 1 but not on the last line, then
\[
C'=AC
\]
for some $A\in O(n)$; that is, $C'$ and $C$ are isometric.
\end{theorem}

\begin{remark}
    Their proof applies not only to sharp arrangements but also to the vertices of a regular 600-cell, though the latter is not relevant for this result.
\end{remark}
To apply Theorem~\ref{cohnkumar theorem}, we rewrite the interaction
\[
\sum_{0\leq k<j<N}
\phi_\alpha(\gamma_j-\gamma_k)
\]
as an energy depending on squared Euclidean distance. We then verify that the resulting potential is completely monotonic. Finally, we use the fact that the $N$th roots of unity form a sharp configuration on $S^1$ by Table 1 in \cite{cohnKumar2007}. This approach actually extends their minimization result to more general weights.

\begin{theorem}\label{extension theorem}
Let $g:[0,1]\to[0,\infty)$ be an integrable function with
$g\not\equiv0$ and satisfies \eqref{g condition}.
Then, for every integer $N\geq1$ and every collection of points
$\{a_k\}_{0\leq k<N}\subset\mathbb{T}$,
\[
\left\|
\sum_{0\leq k<N}\frac{1}{z-a_k}
\right\|_{(g)}
\geq
\left\|
\sum_{0\leq k<N}
\frac{1}{z-e^{2\pi i k/N}}
\right\|_{(g)}.
\]
In particular, for the power weights
\[
g_\alpha(t)=t^\alpha,
\]
the equally spaced configuration uniquely minimizes the $A^2_\alpha$ norm for every $\alpha>0$.
\end{theorem}

\begin{remark}
    It is important to notice that for weights  $g:[0,1]\to[0,\infty)$ satisfying \eqref{g condition}, this result only shows global minimization and is missing uniqueness of the configuration that directly correspond to Abakumov, Borichev, and Fedorovskiy's Theorem \ref{ABF Theorem 1}. However, we are able to extend Theorem \ref{ABF Theorem 1} from $\alpha\in (0,1]$ to any $\alpha>0.$
\end{remark}

\begin{proof}
First define
\[
\phi_\alpha(\theta)
:=\int_0^1
\frac{\cos\theta-s}{1+s^2-2s\cos\theta}g(1-s)\,ds
\]
where $g\not\equiv0$ satisfies \eqref{g condition}. Setting, $t=\cos\theta$, define a new potential function
\[
\psi_{(g)}(t)
:=\int_0^1\frac{t-s}{1+s^2-2st}g(1-s)\,ds,
\qquad t\in(-1,1).
\]

Note that any point on $z_j\in S^1$ can be written as
\[
z_j=(\cos\gamma_j,\sin\gamma_j)
\]
where $\gamma_j\in [0,2\pi).$
Then the inner product between two points on the circle will be
\[
\langle z_j,z_k\rangle=\cos(\gamma_j-\gamma_k),
\]
and
\[
|z_j-z_k|^2
=2-2\langle z_j,z_k\rangle
=2-2\cos(\gamma_j-\gamma_k).
\]

Thus,
\[
\sum_{0\leq k<j<N}\phi_\alpha(\gamma_j-\gamma_k)
=\sum_{0\leq k<j<N}\psi_\alpha(\langle z_j,z_k\rangle).
\]

\begin{lemma}
  Given  $g: [0,1] \rightarrow [0,\infty) $ and $g\not\equiv0$  satisfying  \eqref{g condition}, the function $\widetilde{\psi}_{(g)}(t):= \psi_{(g)}(t)-\psi_{(g)}(-1) $ is absolutely monotonic on $[-1,1).$
\end{lemma}
\begin{proof}
    We wish to show $\widetilde{\psi}_{(g)}^{(k)}(t)\geq 0$ for all $k\geq 0.$  Let $h_s(t):=\frac{t-s}{1+s^2-2st} $ on $[-1,1).$ Then 
    \[h_s^{(k)}(t)
=k!(2s)^{k-1}\frac{1-s^2}{(1+s^2-2st)^{k+1}}
\geq 0 \qquad \text{ for } k\geq 1, \; s\in (0,1), \;t\in [-1,1).\]

Then we have \[\psi_{(g)}^{(k)}(t)=\int_0^1 h_s^{(k)}(t) \;g(1-s) \;ds\geq 0 \qquad \text{ for } k\geq 1.\]
So since $h_s'(t)\geq 0 $ on $[-1,1)$, we know that $h_s$ is increasing on the interval. 
Now see that
$$ \begin{aligned} \psi_g(-1) &= \int_0^1 \frac{-1-s}{1+s^2+2s}\,g(1-s)\,ds\\ &= -\int_0^1\frac{g(1-s)}{1+s}\,ds. \end{aligned} $$

Since \(g\ge0\) and \(g\not\equiv0\), we have

$$\psi_g(-1)<0. $$

Because $\psi_g'(t)>0, $ we see that $\psi_g$ is strictly increasing, so

$$ \psi_g(t)\ge\psi_g(-1). $$

So then $$\widetilde{\psi}_{(g)}(t):= \psi_{(g)}(t)-\psi_{(g)}(-1)\geq0$$ and $$ \widetilde{\psi}_{(g)}^{(k)}(t)= \psi_{(g)}^{(k)}(t)\geq 0 \qquad \text{ for }  k\geq 0 .$$
\end{proof}

Note that since every $N$-point configuration contains the same number of pairwise interactions, adding or subtracting a constant from the potential changes every energy by the same constant. Thus replacing $\psi_{(g)}$ by $\widetilde{\psi}_{(g)}$ does not change the minimizing configurations. Now consider the change of variables 
\[
|x-y|^2=2-2\langle x,y\rangle=2-2t
\quad\Longrightarrow\quad
t=\frac{2-|x-y|^2}{2}=1-\frac{1}{2}|x-y|^2.
\]

So define a new function:
\[
f_{(g)}\bigl(|x-y|^2\bigr)
:=\widetilde{\psi}_{(g)}\left(1-\frac12|x-y|^2\right).
\]
For notation purposes, set $r=|x-y|^2 \in (0,4].$
Then

\[(-1)^kf_{(g)}^{(k)}(r)=(-1)^k \frac{d^{(k)}}{dr^{(k)}}\widetilde{\psi}_{(g)} \left(1-\frac{1}{2}r \right)= (-1)^{2k}k!\frac12
\int_0^1
\frac{s^{k-1}(1-s^2)g(1-s)}{\left(1+s^2-2s(1-\frac12r)\right)^{k+1}}\,ds\]
is completely monotonic.

Finally, Table~1 of \cite{cohnKumar2007} shows that the regular $N$-gon, equivalently the configuration
\[\mathcal{C}=\left\{e^{2\pi i k/N}:k=0,\dots,N-1\right\},\]
is a sharp configuration on $S^1$.

Therefore, the hypotheses of Theorem \ref{cohnkumar theorem} are satisfied.

Hence, for every configuration
\[
\mathcal{C}'\subset S^1
\qquad\text{with}\qquad
|\mathcal{C}'|=|\mathcal{C}|=N,
\]
we have
\begin{align*}
    \sum_{\substack{x,y\in\mathcal{C}'\\x\neq y}}
f_{(g)}\bigl(|x-y|^2\bigr)
&\geq
\sum_{\substack{x,y\in\mathcal{C}\\x\neq y}}
f_{(g)}\bigl(|x-y|^2\bigr)\\
\iff \sum_{\substack{x,y\in\mathcal{C}'\\x\neq y}}
\widetilde{\psi}_{(g)}\bigl(t\bigr)
&\geq
\sum_{\substack{x,y\in\mathcal{C}\\x\neq y}}
\widetilde{\psi}_{(g)}\bigl(t\bigr)\\
\iff \sum_{\substack{x,y\in\mathcal{C}'\\x\neq y}}
\psi_{(g)}\bigl(t\bigr)
&\geq
\sum_{\substack{x,y\in\mathcal{C}\\x\neq y}}
\psi_{(g)}\bigl(t\bigr).
\end{align*}

Thus the equally spaced configuration minimizes the corresponding pairwise interaction energy for distinct $N$ points, and therefore minimizes the associated weighted Bergman norm.

It remains to note that Theorem \ref{cohnkumar theorem} is stated for configurations consisting of distinct points, while we allow the points \(a_k\) to coincide. As in the boundary argument of Section 2.3, this case follows by continuity. Indeed, we assume
$$
\psi_{(g)}(1)=\int_0^1 \frac{g(1-s)}{1-s}\,ds= \int_0^1 \frac{g(u)}{u}\,du<\infty
$$
so then \(f_{(g)}(r)=\widetilde{\psi}_{(g)}(1-r/2)\) extends continuously to $r=0$. Any configuration containing coincident points can be approximated by configurations of $N$ distinct points on $S^1$. Applying Theorem \ref{cohnkumar theorem} to the approximating configurations and passing to the limit shows that the same energy inequality holds when points coincide.

\begin{lemma}\label{alpha is strictly abs mon}
     Given  $g_\alpha:[0,1]\rightarrow [0,\infty)$ defined by $g_\alpha(t)=t^\alpha$.  Then the function $\widetilde{\psi}_{(g_\alpha)}(t):= \psi_{(g_\alpha)}(t)-\psi_{(g_\alpha)}(-1) $ is strictly absolutely monotonic on $(-1,1).$
\end{lemma}

 \begin{proof}
We must show $\widetilde{\psi}_{(g_\alpha)}^{(k)}(t)> 0$ for $k\geq0.$ 

For $k\geq 1$, we already have
$$ \psi_{(g_\alpha)}^{(k)}(t) = k!2^{k-1} \int_0^1 \frac{s^{k-1}(1-s^2)(1-s)^\alpha} {(1+s^2-2st)^{k+1}} \,ds. $$

For
$$ \alpha>0,\qquad -1<t<1,\qquad 0<s<1, $$
every factor in the integrand is strictly positive:

$$ s^{m-1}>0,\qquad 1-s^2>0,\qquad (1-s)^\alpha>0, $$
and
$$ 1+s^2-2st>0. $$

Therefore,
\[\widetilde{\psi}_{(g_\alpha)}^{(k)}(t)> 0 \quad \text{ for } k\geq 1.\]

Now we are left with the $k=0$ case. Since

$$ \psi_{(g_\alpha)}'(t)>0, $$

we have that $\psi_\alpha$ is strictly increasing. Hence for every $t>-1$,

$$ \psi_{(g_\alpha)}(t)>\psi_{(g_\alpha)}(-1). $$

Therefore \[\widetilde{\psi}_{(g_\alpha)}(t)= \psi_{(g_\alpha)}(t)-\psi_{(g_\alpha)}(-1)>0 \quad \text{ for } -1<t<1.\]

 \end{proof}
 
\begin{lemma}
Let $\alpha>0.$ Then let $f_\alpha=f_{(g_\alpha)}$ and so \[f_\alpha\bigl(|x-y|^2\bigr)
:=\widetilde{\psi}_\alpha\left(1-\frac12|x-y|^2\right)\]
is strictly completely monotonic on $0<|x-y|^2<4.$
\end{lemma}
The proof follows directly from Lemma \ref{alpha is strictly abs mon}. 
Then Theorem \ref{cohnkumar theorem} implies that the regular $N$-gon is the unique minimizing configuration, up to rotations and reflections of $S^1$.
\end{proof}

\section*{AI Disclosure}
The author used ChatGPT as an assistive tool to check the algebraic consistency in the kernel analysis. The underlying proof strategies were developed by the author who take full responsibility for the content of the manuscript.

\section*{Acknowledgments}

This material is based upon work supported by the National Science Foundation Graduate Research Fellowship Program under Grant No. DGE 2139839. Any opinions, findings, and conclusions or recommendations expressed in this material are those of the author(s) and do not necessarily reflect the views of the National Science Foundation.  

The author would also like to thank their advisor Dr. Brett D. Wick for suggesting the problem as well as their support and guidance throughout the project. 
\bibliographystyle{plain}
\bibliography{references} 

\end{document}